\documentclass{amsart}
\usepackage[T1]{fontenc}
\usepackage[latin1]{inputenc}
\usepackage{latexsym,amssymb,amsmath}
\usepackage{amsthm}
\usepackage{longtable}
\usepackage{epsfig}
\usepackage{hhline}
\usepackage{epic}

\newcommand{\Galg}{\mathbf{G}}

\newcommand{\Ext}{\operatorname{Ext}}

\newcommand{\F}{\mathbb{F}}
\newcommand{\Fq}{\mathbb{F}_q}

\newcommand{\Z}{\mathbb{Z}}

\newcommand{\cH}{\mathcal{H}}

\newcommand{\SL}{\operatorname{SL}}

\newcommand{\uM}{\underline{M}}
\newcommand{\tm}{\texttt{map}}
\newcommand{\eProof}{\hfill$\Box$}
\newcommand{\GAP}{{\textsc GAP}}

\newtheorem{theorem}{Theorem}[section] 
\newtheorem{lemma}[theorem]{Lemma}

\newtheorem{proposition}[theorem]{Proposition}

\newtheorem{definition}[theorem]{Definition}
\newtheorem{remark}[theorem]{Remark}
\newtheorem{algorithm}[theorem]{Algorithm}
\newtheorem{conj}[theorem]{Conjecture}

\usepackage[pdftex,colorlinks=true,breaklinks=true]{hyperref}

\title[Computation of Kazhdan-Lusztig-polynomials]{Computation of
Kazhdan-Lusztig polynomials and some applications to finite groups}

\author{Frank Lübeck}
\address{
Lehrstuhl D für Mathematik\\
RWTH Aachen\\
Pontdriesch 14/16\\
D-52064 Aachen\\}
\email{Frank.Luebeck@Math.RWTH-Aachen.De}


\begin{document}

\begin{abstract} 
We  discuss  a  practical  algorithm to  compute  parabolic  Kazhdan-Lusztig
polynomials. As an application  we compute Kazhdan-Lusztig polynomials which
are  needed to  evaluate a  character formula  for reductive  groups due  to
Lusztig.

Some coefficients of these  polynomials have interesting interpretations for
certain finite  groups. We find  examples of finite dimensional  modules for
finite groups  with much higher  dimensional first cohomology group  than in
all previously known cases.

Some of these  examples lead to the construction of  finite groups with many
maximal subgroups, contradicting an old conjecture by G.~E.~Wall.
\end{abstract}

\maketitle

\section{Introduction}\label{s:Intro}

Let  $\Galg$  be  a  connected   reductive  algebraic  group  over  a  field
$k$  of    finite  characteristic  $p$.  A  famous  formula  conjectured  by
Lusztig~\cite{LuCF}  (see Theorem~\ref{LuCF})  expresses the  characters and
dimensions  of  some  simple  $k\Galg$-modules  as  linear  combinations  of
characters and dimensions of certain well known modules. The coefficients in
these linear combinations involve Kazhdan-Lusztig  polynomials parameterized
by pairs of elements of the affine Weyl group of $\Galg$.

However,   these   polynomials  are   not   easy   to  compute,   even   for
groups  of  small   rank.  The  by  far  most   efficient  (and  technically
and   mathematically   sophisticated)   implementation   of   an   algorithm
to   compute   ordinary   Kazhdan-Lusztig   polynomials   is   the   program
\textsc{Coxeter~3}~\cite{Coxeter3} by Fokko du Cloux. The original statement
of Lusztig's character formula uses ordinary Kazhdan-Lusztig polynomials. We
tried  to use  \textsc{Coxeter~3} to  compute the  (finite number  of) these
polynomials needed for the case $\Galg = \textrm{SL}_6(k)$ (computing in the
affine Weyl  group of type $A_5$),  but this was not  successful because the
program needed  too much memory, even  on a modern computer.  The problem is
that  the parameterizing  elements  of  the needed  polynomials  have a  big
Coxeter  length (they  involve  the longest  element of  the  Weyl group  of
$\Galg$).

The first step to  handle this and larger cases is to use  the fact that the
needed polynomials can also be described as \emph{parabolic} Kazhdan-Lusztig
polynomials so  that the parameterizing  elements of the  needed polynomials
are much shorter (see the discussion in~\cite[8.22 and C.2]{Jan03}). But the
computation of these polynomials remains challenging.

In this article we use an elegant description of (parabolic) Kazhdan-Lusztig
polynomials which is due to  Soergel~\cite{Soe97} (these are closely related
to the polynomials originally defined by Kazhdan and Lusztig~\cite{KL79} and
Deodhar~\cite{Deo87}, see the comment after Definition~\ref{def:KL}).

Our  algorithm to  compute these  polynomials  works for  any Coxeter  group
and  parabolic  subgroup.  We  mention  two  of  its  features:  While  most
implementations  (and   the  usual   existence  proof)   of  Kazhdan-Lusztig
polynomials assume inductively that the polynomials for "smaller" parameters
are already  known, we  do not do  so. This allows  for a  trade-off between
memory  usage and  computing   time,  see~\ref{rem:KLalg}(a). (A  variant of
this strategy  was also used  by Scott  and Sprowl in~\cite{SS16}.)  For the
case that  it is  known that  the polynomials  to compute  have non-negative
coefficients, we give an upper bound  for all coefficients which  allows for
modular computations, see~\ref{KLcoeffbound}.

The algorithm  is described in Section~\ref{s:KLcomp}.  The general strategy
is given in  Algorithm~\ref{alguMx}. The crucial ingredient  of an efficient
implementation  is a  data structure  that  encodes the  Bruhat ordering  on
relevant elements. This is detailed in Algorithm~\ref{alg:bruhatmap} and may
be of interest for other applications. In Section~\ref{s:KLapp} we give more
details on the application of  our algorithm to Lusztig's character formula.
We have computed the relevant polynomials  for all simple $\Galg$ up to rank
$6$ and also for types $A_7$ and $A_8$.

In Section~\ref{s:GurConj} we  show in some detail how  some coefficients of
the polynomials computed in Section~\ref{s:KLapp} lead to examples of finite
groups which have a simple  finite dimensional module with large dimensional
first  cohomology  group.  (A  similar  argument  was  already  sketched  by
L.~L.~Scott in~\cite{Scott03}.)  This has  consequences for a  conjecture of
R.~M.~Guralnick,  which with  our examples  in mind  seems unlikely  to hold
(although it is not disproved).

In  the  last  Section~\ref{s:WallConj}  we construct  from  an  example  in
Section~\ref{s:GurConj} an infinite series of  finite groups which have more
maximal subgroups than predicted by an old conjecture of G.~E.~Wall.

\noindent\textbf{Acknowledgements.}   I   learned  about   the   interesting
consequences of  the computations  mentioned in  Section~\ref{s:KLapp} which
are   described  in   Sections~\ref{s:GurConj},~\ref{s:WallConj}  during   a
workshop  on  \emph{Cohomology bounds  and  growth  rates} at  the  American
Institute of Mathematics (AIM) in Palo  Alto. During that workshop Len Scott
made me aware  of the relevance of some  of the data I had  computed and his
work~\cite{Scott03}.  Furthermore,  he  provided me  with  some  non-trivial
Kazhdan-Lusztig  polynomials  in types  $A_6$,  $A_7$  computed by  him  and
Sprowl, so that I could check  if they coincide with polynomials computed by
my own programs (they did).

When Bob  Guralnick saw the  example in type $F_4$  in Theorem~\ref{maxdimH}
(which was  computed during a  presentation at the workshop)  he immediately
noticed that this leads to counterexamples to Wall's conjecture, and he gave
me a sketch of the proof of Theorem~\ref{cexWall}.

I  thank both  of them  for allowing  me  to use  their hints  in the  later
sections of this paper.  I also wish to thank the AIM  and the organizers of
the workshop for inviting me to participate in this inspiring event.

\section{Computing parabolic Kazhdan-Lusztig polynomials}\label{s:KLcomp}
\subsection{Notation and basic facts.}\label{ss:KLnotation}
Let $(W,S)$  be a  Coxeter system with  a finite generating  set $S$  and $J
\subseteq S$. We write $W_J =  \langle J \rangle$ for the parabolic subgroup
of $W$  generated by $J$, and  we write $W^J$  for the set of  reduced right
coset representatives of $W_J \setminus W$.

For $x \in  W$ we write $l(x)$ for  the length of $x$, that  is the smallest
number $n$ such that $x = s_1 s_2 \cdots s_n$ with $s_i \in S$ for $1 \leq i
\leq n$.  Then the sequence  $(s_1, s_2, \ldots,  s_n)$ is called  a reduced
word for $x$.

We write $\leq$  for the Bruhat ordering  of $W$. We will  use the following
recursive definition  of this partial  ordering: $1  \in W$ is  smaller than
every other element in $W$, and for $x \in  W$, $s \in S$ with $x = x's$ and
$l(x') < l(x)$ we  define $\{ y \in W\mid\; y \leq x\}  := \{y \in W\mid\; y
\leq x'\}  \cup \{ ys  \in W\mid\;  y \leq x'\}$.  See~\cite[5.9]{Hum90} for
other characterizations of the Bruhat order and a proof that the ordering is
well defined by our definition.

Observe that  for any  $x \in W$  the set  $\{y \in W\mid\;  y \leq  x\}$ is
finite.
An element $x  \in W$ is in $W^J$ if  and only if $l(sx) > l(x)$  for all $s
\in J$.

\subsection{Parabolic Kazhdan-Lusztig polynomials.}\label{ss:parKL}
Deodhar introduced  in~\cite{Deo87} the notion of  parabolic Kazhdan-Lusztig
polynomials of  $W$ with  respect to  $W_J$. We  recall a  slightly modified
setup introduced by Soergel in~\cite{Soe97}.

Let  $\cH$ be  the  Hecke algebra  of  $W$  over the  ring  $A =  \Z[v,1/v]$
of  Laurent  polynomials in  an  indeterminate  $v$  over the  integers.  As
$A$-algebra $\cH$  is generated by  elements $C_s$, $s  \in S$. There  is an
involution $\bar{\  }: \cH  \to \cH$  which is a  ring homomorphism  with $v
\mapsto 1/v$ and $\bar{C_s} = C_s$.

There is  an $\cH$-module  $M$ (induced  from a linear  module of  the Hecke
algebra  of the  parabolic subgroup  $W_J$,  for $J=\emptyset$  this is  the
regular module of $\cH$) which can be described as follows.

$M$ is  a free $A$-module with  an $A$-basis $\{M_x\mid\; x  \in W^J\}$ such
that the generators $C_s$ of $\cH$ act by
\begin{equation}\label{eq:Mxoper}
M_x C_s = \left\{ \begin{array}{ll}
M_{xs} + v M_x,& \textrm{ if } xs \in W^J, l(xs) > l(x)\\
M_{xs} + v^{-1} M_x,& \textrm{ if } xs \in W^J, l(xs) < l(x)\\
(v + v^{-1}) M_x,& \textrm{ if } xs \notin W^J
\end{array} \right.
\end{equation}
Remark. It is clear that for $x \in  W^J$ with $l(xs) < l(x)$ we always have
$xs \in W^J$. So, the third case can only occur when $l(xs) > l(x)$.

There is an involution $\bar{\ }: M  \to M$ with $\bar{M_1} = M_1$ and which
is compatible with the involution on $\cH$, that is $\overline{mh} = \bar{m}
\bar{h}$ for all $m \in M$ and $h \in H$.

\begin{lemma}\label{lemmaBx}
\mbox{}
\begin{itemize}
\item[(a)] Let $x' \in W^J$ and  $s \in  S$ such that $x = x's \in W^J$ and
$l(x) > l(x')$. 
Let $B_{x'} = \overline{B_{x'}} \in M$  be a $\bar{\ }$-invariant element of
the form
\[ B_{x'} = M_{x'} + \sum_{W^J \ni y < x'} h_y M_y \textrm{ with }
h_y \in A.\] 
Then the elements $B_x := B_{x'} C_s$ and $B'_x := B_{x'} C_s - (v + v^{-1})
B_{x'}$ are also $\bar{\ }$-invariant and of the form
\[\begin{array}{lcl}
B_x &=& M_x + \sum_{W^J \ni y < x} f_y M_y \textrm{ with }f_y \in A, \\
B'_x &=& M_x + \sum_{W^J \ni y < x} g_y M_y \textrm{ with }g_y \in A.
\end{array}\]
More precisely,  (setting $h_{x'}  = 1$) we  have for $y  \leq x'$  with $ys
\notin W^J$, that $y \leq x$ and
\[ f_y = (v+v^{-1})h_y, \quad g_y = 0.\]
And for $y \leq x'$ with $y < ys \in W^J$ we have
\[\begin{array}{ll}
f_y = h_{ys} + v h_y,& f_{ys} = h_y + v^{-1} h_{ys},\\
g_y = h_{ys} - v^{-1} h_y, & g_{ys} = h_y - v h_{ys}.
\end{array}\] 
\item[(b)] For every $x \in W^J$ there exists a $\bar{\ }$-invariant element
$B_x \in M$ of the form
\[ B_{x} = M_{x} + \sum_{W^J \ni y < x} f_y M_y \textrm{ with }
f_y \in A.\]
\end{itemize}
\end{lemma}
\textbf{Proof.}   Part   (a)   is    an   immediate   consequence   of   the
formula~(\ref{eq:Mxoper}).

Part (b) follows by induction on the length  of $x$. If $x=1$ we take $B_x =
M_1$ which is  $\bar{\ }$-invariant. Otherwise let $x=x's \in  W^J$ with $x'
\in  W^J$, $s  \in S$,  and $l(x')  < l(x)$.  By induction  there exists  an
element  $B_{x'}  \in  M$  of  the desired  form.  Now  part~(a)  shows  two
possibilities to find a $B_x$ as desired.
\eProof

It is possible to find elements $B_x$  as in the lemma with more restrictive
conditions  on the  coefficients $f_y$.  This is  detailed in  the following
theorem.

\begin{theorem}\label{thm:KL}
\begin{itemize}
\item[(a)] For  any $x \in W^J$  there exists a unique  $\bar{\ }$-invariant
element
\[ \uM_x = \overline{\uM_x} = \sum_{y \in W^J} m_{y,x} M_y\]
with  $m_{x,x} = 1$ and $m_{y,x} \in v \Z[v]$ if $x \neq y$.
\item[(b)] If in~(a) the polynomial $m_{y,x} \neq 0$, then $y \leq x$.
\end{itemize}
\end{theorem}

\textbf{Proof.}   This  is   ~\cite[Theorem   3.1]{Soe97}.  The   uniqueness
follows   easily    from   the   existence    (see~~\cite[2.4]{Soe97}).   In
Algorithm~\ref{alguMx}  below  we show  how  to  compute $\uM_x$  as  in~(a)
and~(b)  of the  theorem as  $A$-linear combination  of elements  $B_y$ from
Lemma~\ref{lemmaBx}(b). This yields an alternative constructive proof of the
existence of $\uM_x$.
\eProof

\begin{definition}\label{def:KL} \rm 
In this paper we call the polynomials  $m_{y,x} \in \Z[v]$ for $x,y \in W^J$
in Theorem~\ref{thm:KL} the  \emph{parabolic Kazhdan-Lusztig polynomials} of
$W$ with respect  to $W_J$. In the  case $J = \emptyset$ they  are also just
called (ordinary) \emph{Kazhdan-Lusztig polynomials}.

But  note  that  the   (parabolic)  Kazhdan-Lusztig  polynomials  originally
defined  by Kazhdan-Lusztig  and  Deodhar are  slightly  different: You  get
them  from  the $m_{y,x}$  as  $v^{l(x)-l(y)}  m_{y,x}(1/v)$, considered  as
polynomials in $v^2$. See~\cite[2.6 and 3.2.(1.)]{Soe97}. In particular, the
Kazhdan-Lusztig $\mu(y,x)$-function  is given by  the coefficient of  $v$ in
$m_{y,x}$. We do not need this function in our setup.
\end{definition}

We can write any Laurent polynomial $f  \in A = \Z[v,1/v]$ in the form $f(v)
= f^-(1/v) + f^0 + f^+(v)$ with $f^0 \in \Z$ and $f^+, f^- \in v\Z[v]$. Then
$f^\textrm{sym}(v) := f^-(1/v)  + f^0 + f^-(v)$ is  $\bar{\ }$-invariant and
$f(v) - f^\textrm{sym}(v) \in v\Z[v]$. With  this notation we can describe a
computation of the $\uM_x$ in Theorem~\ref{thm:KL}.

\begin{algorithm}\label{alguMx} \rm
\textbf{Input:} An element $x \in W^J$.

\textbf{Output:} The element $\uM_x$ from Theorem~\ref{thm:KL}.

(1) Initialize $\uM_x := B_x = M_x + \sum_{W^J \ni y < x} f_y M_y$ with 
any $\bar{\ }$-invariant element as in Lemma~\ref{lemmaBx}(b).

(2) Find a $y \in W^J$ of maximal length such that $f_y \notin v\Z[v]$.

(3) If in~(2) no such $y$ is found, then $\uM_x$ has the desired form and 
we return it.

(4) Otherwise compute a $B_y$ as in Lemma~\ref{lemmaBx}(b), substitute 
$\uM_x := \uM_x - f^\textrm{sym}_y B_y$ and return to step~(2). 
\end{algorithm}

\textbf{Proof.}  It is  clear  that in  step~(4) the  new  $\uM_x$ is  again
$\bar{\ }$-invariant and that the new coefficient of $M_y$ is in $v\Z[v]$.

Furthermore, in step~(4)  only the coefficient of $M_y$  and coefficients of
$M_z$ where $l(z) < l(y)$ are changed (and the coefficients of $M_z$ with $z
\neq y$ and $l(z) \geq l(y)$ are not changed).

This  shows that  the  algorithm  will terminate  after  at  most $|\{y  \in
W^J\mid\; y < x\}|$ loops.
\eProof

\begin{remark}\label{rem:KLalg} \rm
\mbox{}
\begin{itemize}
\item[(a)] Assume that for $x\in W^J$ we have  $x = x's$, $s \in S$, $x' \in
W^J$, $l(x')  < l(x)$.  Then, in step~(1)  of Algorithm~\ref{alguMx}  we can
initialize  $\uM_x$ with  $\uM_{x'} C_s  = M_x  + \sum_{  y <  x} f_y  M_y$.
Lemma~\ref{lemmaBx}(a) shows that in this case all $f_y$ are polynomials, so
the $f_y^\textrm{sym}(v) = f_y(0)$  are constant polynomials. If furthermore
we  choose in  step~(4)  always $B_y  = \uM_y$,  then  the substitutions  of
$\uM_x$  in that  step will  never change  the absolute  term of  any $f_y$.
Therefore, we have  the formula \[ \uM_x = \uM_{x'}  C_s - \sum_{y<x} f_y(0)
\uM_y. \]  This yields  the existence  proof of  Kazhdan-Lusztig polynomials
which is usually found in the literature.

In the  other extreme, we  could start with $B_1  = M_1$ and  use throughout
the  Algorithm~\ref{alguMx} elements  $B_x$  as  constructed recursively  in
Lemma~\ref{lemmaBx}(a). In  that case we  would end up with  $\uM_x$ without
any need of Kazhdan-Lusztig polynomials $m_{y,z}$ for $z \neq x$. (This
strategy was used in~\cite{SS16} to compute an interesting Kazhdan-Lusztig
polynomial in type $A_8$, see remark before Theorem~\ref{maxdimH}.)

Alternatively,  we   can  use   the  $B'_x$  instead   of  the   $B_x$  from
Lemma~\ref{lemmaBx}(a). This leads to much sparser elements and saves memory
during the computation.

We can  also store some  already computed  $B_{x'}$ or $\uM_{x'}$  and reuse
them during the computation  of some $\uM_x$ with $x' <  x$. This shows that
our Algorithm~\ref{alguMx} is  quite flexible, it allows to  trade memory or
disk space against computing time.

\item[(b)] In step~(2) of Algorithm~\ref{alguMx} we can also search all such
$y$ of the same maximal length together  and then compute in step~(4) all of
the $f^\textrm{sym}_yB_y$  for these  $y$ in  parallel (on  a multiprocessor
machine, or on several machines).

\item[(c)] It is easy to  see that in all $B_x = M_x +  \sum_{W^J \ni y < x}
f_y M_y$  constructed in  Lemma~\ref{lemmaBx} the Laurent  polynomials $f_y$
and so also the $m_{y,x}$ have the  form $v^a$ times a Laurent polynomial in
$v^2$, where $a \equiv l(x)-l(y) \pmod 2$ (that is, every second coefficient
is  zero). This  can  be used  in  implementations to  save  memory for  the
polynomials.
\end{itemize}
\end{remark}

For an  implementation of Algorithm~\ref{alguMx} and  the recursion formulae
in~\ref{lemmaBx} it is crucial to find  a way to efficiently encode elements
$B_x$  and to  efficiently evaluate  the recursion  in~\ref{lemmaBx}(a). For
this we now define an appropriate data structure.

\begin{definition}\label{defmap}\rm
We fix an ordering  of $S = \{s_1,\ldots,s_r\}$. Let $x \in  W^J$ and $I_x =
\{ y \in  W^J \mid\; y \leq x  \}$ the Bruhat interval of  elements in $W^J$
between $1$  and $x$. A  \emph{Bruhat map of  $I_x$} is a  list \texttt{map}
whose entries  are in  bijection with  $I_x$, the entries  are lists  of $r$
integers, such that we have:
\begin{enumerate}
\item[(A)] 
If $y, y' \in I_x$ and $y <  y'$ then $y$ corresponds to an earlier position
in \texttt{map}  (that is, \texttt{map}  encodes an ordering of  $I_x$ which
refines the Bruhat order).
\item[(B)] 
If $y$  corresponds to the $i$-th  entry \texttt{map[i]} and $1  \leq j \leq
r$,  then the  $j$-th entry  of \texttt{map[i]}  is either  the position  in
\texttt{map} corresponding  to the element  $ys_j$ if  $ys_j \in I_x$  or an
integer $\leq 0$ otherwise.
\end{enumerate}
\end{definition}

If in such a Bruhat map of $I_x$  the element $y \in W^J$ corresponds to the
$i$-th entry, then we  can easily read off from this  $i$-th entry for which
$s \in S$ the  element $ys$ is in $I_x$ and in that  case if $ys$ is shorter
or longer  than $y$.  In particular, by  a simple recursion,  we can  find a
reduced word for $y$ and the length of $y$.

With respect to  this Bruhat map, it  is easy to encode an  element $B_y$ as
in~\ref{lemmaBx} as a list of pairs,  where a summand $f_z M_z$ is specified
by a pair  consisting of the index of \texttt{map}  corresponding to $z$ and
the  Laurent polynomial  $f_z \in  A$. It  is obvious  how to  add two  such
expressions and it is easy to  compute a product $B_yC_s$ using the formulae
in~\ref{lemmaBx}(a).

In the next  algorithm we describe how  to compute a Bruhat map  for a given
element $x \in  W^J$. (A similar, but more sophisticated,  algorithm for the
special case $J=\emptyset$ was described by du Cloux in~\cite{duCloux02}.)

We assume that we can invert and  multiply elements of the Coxeter group $W$
and  that for  $w \in  W$ we  can determine  its left  descent set  $\{s \in
S\mid\; l(sw) <  l(w)\}$. If $s$ is in  the left descent set of  $w$, we can
write $w = s  (sw)$ and then $l(sw) < l(w)$; this way  we find recursively a
reduced word for any $w \in W$  and the length of $w$. This functionality is
for example provided by the CHEVIE software package, see~\cite{CHEVIE}.

\begin{algorithm}\label{alg:bruhatmap}
\rm
\textbf{Input:} A  Coxeter group $W$ with  generators $S$ (given in  a fixed
ordering) in some  representation which allows to compute  left descent sets
of elements, a subset $J \subseteq S$ and an $x \in W^J$.

\textbf{Output:}   A    Bruhat   map    \texttt{map}   of   $I_x$    as   in
definition~\ref{defmap}.

(1) We compute a reduced word $(s_1, s_2, \ldots, s_k)$ for $x$.

(2) We  initialize the  list \tm\ with  its first  entry \texttt{\tm[1]}$:=$
\texttt{[0,\ldots,0]}.

(3) We keep during the algorithm a list \texttt{reps} of the elements in the
Bruhat  interval  $I_x$, we  initialize  \texttt{reps[1]}  with the  trivial
element of $W$.

(4) Furthermore, we keep  for each element $y \in I_x$ a  list of indices of
the elements $\{w \in W^J\mid\; w  < y, l(w) = l(y)-1\}$ (the \emph{coatoms}
of $y$), we initialize \texttt{coatoms[1]} as the empty list.

Now assume that we have collected in \tm, \texttt{reps} and \texttt{coatoms}
the information for the $n$ elements in the Bruhat interval $I_{x'}$ for $x'
= s_1s_2\cdots s_{i-1}$ with some $i\leq  k$ ($i = 1$ after initialization).
We show how to extend the data structures to $I_{x's_i}$.

(5) For  \texttt{j} from  $1$ to  $n$ we consider  $y :=  $ \texttt{reps[j]}
$s_i$. If the entry  of $s_i$ in \texttt{\tm[j]} is not zero,  then $y$ is a
smaller element that is already in our list.

Otherwise we compute  $y$. If the left descent set  of $y$ contains elements
from $J$ then $y \notin W^J$ and we change the entry of \texttt{maps[j]} for
$s_i$ to $-1$.

(5a) Otherwise  $y$ is  a new  element in  $I_{x's_i}$ and  we append  it to
\texttt{reps}. We  can record  the entry for  $s_i$ in  \texttt{\tm[j]}, and
we  know the  \tm-entry  for  $y$ corresponding  to  $s_i$,  the others  are
initialized to zero.

(5b)  The coatoms  of  the  new element  $y$  are  \texttt{reps[j]} and  the
elements $w  s_i$ where $w$ is  a coatom of \texttt{reps[j]}  and $l(ws_i) >
l(w)$ (we can decide this condition because we have already considered those
$ws_i$ in this  loop). So, we can easily add  the \texttt{coatoms} entry for
$y$.

(6) Now  we consider for all  \texttt{j }$> n$ the  list \texttt{\tm[j]} and
all  $s \in  S\setminus  \{s_i\}$. If  $y :=  $  \texttt{reps[j]}$\cdot s  <
$\texttt{ reps[j]} then $y$ must be  one of the coatoms of \texttt{reps[j]}.
In  this case  we  adjust \texttt{\tm[j]}  and the  \tm-entry  for $y$  with
respect to $s$.

(7)  Optionally, we  can save  memory by  deleting all  \texttt{reps[j]} and
\texttt{coatoms[j]} for those  $j$ for which all  entries of \texttt{\tm[j]}
are positive or $-1$, these are no longer needed.

(8) If the information for $I_x$  is complete (above steps are completed for
$i=k$), return \tm. Otherwise increase $i$ by one and return to step~(5).
\end{algorithm}

\textbf{Proof.} From our definition of the  Bruhat order it is clear that we
find the Bruhat  interval $I_x$ as the intersection of  $W^J$ with the union
of $I_{x'}$  and $I_{x'}s$ if $l(x')  < l(x)$ and  $x's = x$. It  follows by
induction that  after finishing steps~(5)-(6) for  some $i \leq k$  the list
\texttt{reps} contains the elements of $I_{x'}$ for $x' = s_1\cdots s_i$.

That  we find  in step~(5b)  the  correct set  of coatoms  follows from  the
following  characterization of  the coatoms:  if $(t_1,  \ldots, t_k)$  is a
reduced  word for  $y$ then  the  coatoms of  $y$ are  the partial  products
$t_1\cdots  \hat{t_j} \cdots  t_r$ ($t_j$  left  out), which  are of  length
$r-1$. We refer to~\cite[5.9]{Hum90} for more details.

It  is easy  to check,  that in  all  steps of  our algorithm  we only  need
information which was computed (and not deleted) before.
\eProof

The following  lemma can be  useful for  implementations of an  algorithm to
compute  Kazhdan-Lusztig  polynomials.  Instead of  computing  with  Laurent
polynomials over  the integers  one could  compute with  Laurent polynomials
over rings  $\Z/m_i\Z$ for some  mutually coprime  numbers $m_i \in  \Z$ and
reconstruct the integer polynomials with the Chinese remainder theorem.

In  many  cases  it  is  known  that  the  coefficients  of  Kazhdan-Lusztig
polynomials   are   non-negative.   For  example,   whenever   $J=\emptyset$
by~\cite{EW12},  and  so also  whenever  $W_J$  is  finite because  in  this
case  parabolic  Kazhdan-Lusztig polynomials  are  equal  to ordinary  ones,
see~\cite[3.4]{Soe97}.  (See also  the  interpretation  of the  coefficients
in~\ref{thm:KLextdim}.) If  we know an  upper bound for the  coefficients of
Kazhdan-Lusztig  polynomials then  we can  reconstruct these  polynomials by
computing them modulo several coprime  numbers $m_i$ such that their product
is larger than this bound.

\begin{lemma}\label{KLcoeffbound}
Assume  that   for  $W$,  $J$  the   parabolic  Kazhdan-Lusztig  polynomials
$m_{y,x}(v)$  for all  $x,y \in  W^J$ have  non-negative coefficients.  Then
$2^{l(x)}$ is an upper bound for the coefficients of $m_{y,x}(v)$.
\end{lemma}
\textbf{Proof.}
We use  induction on  the length $l(x)$  of $x$. If  $l(x)=0$ then  the only
nonzero polynomial is $m_{x,x} = 1$. Otherwise, let $x = x's$ with $s \in S$
and  $l(x') =  l(x)-1$. In  our construction  of $\uM_x$  we can  start with
$\uM_{x'} = \sum_{y \leq x'} m_{y,x'}  M_y$. By the induction hypothesis the
$m_{y,x'}$ have non-negative coefficients  which are bounded by $2^{l(x')}$.
From the formulae in~\ref{lemmaBx}(a) we see that $\uM_{x'}C_s = \sum_{y\leq
x} f_y M_y$ where all $f_y \in \Z[v]$ are polynomials whose coefficients are
non-negative and bounded by $2\cdot 2^{l(x')} = 2^{l(x)}$ (they are all sums
of two polynomials of the form $v^i m_{y,x'}$).

But  this upper  bound  also holds  for the  coefficients  of the  $m_{y,x}$
because $\uM_x =  \sum_{y\leq x} m_{y,x} M_y  = \uM_{x'} C_s -  \sum_{y < x}
f_y(0)  \uM_y$ (the  $f_y(0) \uM_y$  are linear  combinations of  terms $g_z
M_z$, $z \in W^J$, where the $g_z$ have non-negative coefficients).
\eProof

\section{Application to Lusztig's character formula}\label{s:KLapp}

We sketch an important application of parabolic Kazhdan-Lusztig polynomials.
References and more details can be found in Jantzen's book~\cite{Jan03}.

Let  $\Galg$   be  a  connected   reductive  group   of  rank  $l$   and  of
simply-connected  type  over an  algebraically  closed  field $k$  of  prime
characteristic $p$.

Let $X \cong  \Z^l$ be the weight  lattice of $\Galg$, $\Phi  \subset X$ the
set  of  roots  of  $\Galg$,  and  $\{\omega_1,  \ldots,  \omega_l\}$  be  a
$\Z$-basis of  fundamental weights of  $X$. This  basis determines a  set of
simple roots  of $\Phi$ and this  induces a certain partial  order $\leq$ on
$X$.

The  simple  finite-dimensional  rational  modules  of  $\Galg$ over $k$ are
parameterized by the set of dominant weights $X^+ = \sum_{i=1}^l \Z_{\geq 0}
\omega_i$. For  $\lambda \in X^+$  we denote $L(\lambda)$  the corresponding
simple module and $\chi(L(\lambda))$ its character. In general the dimension
and character of $L(\lambda)$ is not known.

For  each  $\lambda  \in  X^+$   there  is  a  module  $V(\lambda)$,  called
the  \emph{Weyl  module}  of  highest weight  $\lambda$.  The  dimension  of
$V(\lambda)$  is   easy  to  compute   from  $\lambda$  and   the  character
$\chi(V(\lambda))$ can also be computed.

For the rest of this section we assume  that $p$ is at least as large as the
Coxeter number of the root system $\Phi$. The Weyl group $W$ of $\Galg$ acts
faithfully on  $X$, it  has a  set $S$  of Coxeter  generators which  act as
reflections on $X$. Let  $W_p = p\Z\Phi \rtimes W$ be  the affine Weyl group
of $\Galg$. It is also a Coxeter group with a set $S'$ of Coxeter generators
such that $S \subset S'$. The  isomorphism type of the Coxeter system $(W_p,
S')$ does  not depend on  $p$. The affine Weyl  group $W_p$ has  a \emph{dot
action} on $X$  (written $w.\lambda$ for $w \in W_p$,  $\lambda \in X$). The
orbit of $0 \in X$ under $W_p$ is  regular and $w.0$ is dominant if and only
if $w$  is a reduced  right coset representative  of the finite  subgroup $W
\leq W_p$ (that is $x \in W_p^S$ in the notation of Section~\ref{s:KLcomp}).

Let 
\[B = \{ w \in W_p\mid\; w.0 = \sum_{i=1}^l a_i \omega_i \textrm{ with }
0 \leq a_i < p \textrm{ for }i=1,\ldots, l\},\]
the  set of  elements in  $W_p$  which map  $0\in X$  to the  $p$-restricted
region. Modulo the isomorphisms mentioned above between the groups $W_p$ for
all  primes  $p$  the  set  $B$  is independent  of  $p$.  We  have  $|B|  =
|W|/[X:\Z\Phi]$, there is a longest $w' \in  B$ for which $w'.0$ is close to
the Steinberg  weight $(p-1) (\sum_{i=1}^l  \omega_i)$, and we have  $w \leq
w'$ for all $w \in B$ (see~\cite[4.9]{Hum90}).

The  \emph{translation principle}  and  the  \emph{Steinberg tensor  product
theorem}  reduce  the computation  of  $\chi(L(\lambda))$  for all  $\lambda
\in  X^+$ to  the  finite number  of cases  $L(\lambda)$  with $\lambda  \in
\{w.0\mid\; w  \in B\}$ (the  \emph{$p$-restricted} dominant weights  in the
$W_p$-orbit of $0$).

It follows from the \emph{linkage principle}  that for $\lambda \in X^+$ the
character  $\chi(L(\lambda))$  is a  $\Z$-linear  combination  of the  known
characters  $\chi(V(\mu))$  for the  finite  number  of  $\mu \in  X^+  \cap
W_p.\lambda$ with $\mu \leq \lambda$.

For large  enough $p$ this linear  combination can be expressed  in terms of
parabolic Kazhdan-Lusztig polynomials $m_{y,x}(v)$ for the affine Weyl group
$(W_p,S')$ with  respect to  the finite  Weyl group $W  = (W_p)_S$,  that is
$J=S$ in the notation of Section~\ref{s:KLcomp}.

\begin{theorem}[Andersen-Jantzen-Soergel, Fiebig]\label{LuCF}
Fix the  type of  a root  system $\Phi$. For  any prime  $p$ let  $B \subset
W_p$  be  the  set  defined  above.  Then for  almost  all  primes  $p$  the
characters $L(x.0)$ with  $x \in B$ of  a group $\Galg$ over a  field $k$ of
characteristic $p$ and with root system $\Phi$ are given by
\[ \chi(L(x.0)) = \sum_{y\in W_p, y.0 \in X^+} m_{y,x}(-1) \chi(V(y.0)).\]
\end{theorem}

\noindent
\textbf{Remarks.} This formula was  first conjectured by Lusztig~\cite{LuCF}
(with the Coxeter number  of $\Phi$ as an upper bound  for the finite number
of exceptional primes for which the  formula does not hold). The theorem was
first shown  in Andersen,  Jantzen and  Soergel~\cite{AJS94} but  without an
explicit upper bound for the exceptional primes. An explicit (although huge)
bound for  the exceptional  primes in  terms of the  root system  $\Phi$ was
given by Fiebig~\cite{F12}. More recently, Williamson~\cite{W13} showed that
there exist series  of examples of groups $\Galg$ and  primes $p$, where $p$
grows exponentially  in terms of  the Coxeter number  of the root  system of
$\Galg$, such that the formula in the theorem does not hold.

\subsection{Computations.}\label{ss:Computations}
As an  application of the  algorithm described in  Section~\ref{s:KLcomp} we
have  computed  all parabolic  Kazhdan-Lusztig  polynomials  which occur  in
Theorem~\ref{LuCF} for  groups of small  rank. To  indicate the size  of the
computations we list the types, the sizes of the set $B$, and the size of
\[ B_\leq := \{ y \in W_p\mid\; y.0 \in X^+, y.0 \leq x.0 \textrm{ for some }
x \in B\}\]
which is computed with Algorithm~\ref{alg:bruhatmap}. In the rows labeled by
$m$ we give the maximal length of an element in $B_\leq$ (these lengths were
also determined in~\cite{Boe98}).
\begin{center}
\begin{tabular}{l|ccccccc}
\hline
Type & $A_2$ & $A_3$ & $A_4$ & $A_5$ & $A_6$ & $A_7$ & $A_8$ \\
$|B|$ & 2 & 6 & 24 & 120 & 720 & 5040 & 40320 \\
$|B_\leq|$ & 2 & 8 & 52 & 478 & 5706 & 83824 & 1461944 \\
$m$& 1 & 4 & 10 & 20 & 35 & 56 & 84 \\
\hline
\end{tabular}

\medskip
\begin{tabular}{l|ccccccccc}
\hline
Type & $B_2$ & $B_3$ & $B_4$ & $B_5$ & $B_6$ & $C_3$ & 
$C_4$ & $C_5$ & $C_6$ \\
$|B|$ & 4 & 24 & 192 & 1920 & 23040 & 24 & 192 & 1920 & 23040 \\
$|B_\leq|$ & 4 & 44 & 756 & 17332 & 493884 &
46 & 792 & 17958 & 504812\\
$m$ & 3 & 13 & 34 & 70 & 125 & 13 & 34 & 70 & 125 \\
\hline
\end{tabular}

\medskip
\begin{tabular}{l|cccccccccc}
\hline
Type & $D_4$ & $D_5$ & $D_6$ & $G_2$ & $F_4$ & $E_6$ \\
$|B|$ & 48 & 480 & 5760 & 12 & 1152 & 17280 \\
$|B_\leq|$ & 142 & 3428 & 102142 & 16 & 7832 & 437230 \\
$m$ & 16 & 40 & 80 & 10 & 86 & 120 \\
\hline
\end{tabular}
\end{center}

The next  case $E_7$ is  much bigger and  the computations of  the parabolic
Kazhdan-Lusztig polynomials could not yet be  finished, in this case we have
$|B|=1451520$, $|B_\leq|=139734574$, $m = 336$.

We used a  prototype implementation in \GAP~\cite{GAP4.7.6} and  C for these
computations. For example,  the case of type $F_4$ needs  about 6~seconds of
CPU time (on a current desktop computer), while the case of type $E_6$ needs
about 20~days of CPU time (with parallel processes we could finish this case
in about 3~days) and the result needs about 400~GB of disk storage.

We give another example for the performance of the algorithm: Computing  all
ordinary  Kazhdan-Lusztig  polynomials for  finite  Weyl  groups $W$  ($J  =
\emptyset$)  is  also  possible,  for  example for  type  $F_4$  this  takes
0.2~seconds and for type $E_6$ about 16~minutes.

\section{First cohomology and Guralnick's conjecture}\label{s:GurConj}

Let  $k$ be  a  field, $H$  be  a group  and $M$  be  a $kH$-module.  Recall
that  the first  cohomology $H^1(H,M)$  is defined  as the  quotient of  the
$k$-vectorspace of maps
\[\{f:H\to M\mid\; f(hh')=f(h)+hf(h') \textrm{ for all }h,h'\in H\}\]
modulo the subspace $\{f:H\to M, h\mapsto hm-m\mid\; m \in M\}$.

We  will  use  two  interpretations   of  $H^1(H,M)$.  First,  it  yields  a
parameterization  of extensions  of $M$  with  the trivial  module, that  is
$H^1(H,M) \cong {\Ext}^1_G(k, M)$.

For  a second  interpretation consider  $G =  M \rtimes  H$, the  semidirect
product where  the action of $H$  on $M$ is  given by the module  action. We
have the following lemma, see~\cite[Prop. 3.7.2]{BenCohI} for a proof.

\begin{lemma}\label{ComplClasses}
Let $k$, $H$, $M$,  $G = M \rtimes H$ as before.  The elements of $H^1(H,M)$
parameterize the conjugacy classes of complements of the normal subgroup $M$
in $G$.
\end{lemma}

From now  we consider finite  groups $H$ and  finite-dimensional irreducible
$kH$-modules $M$. We  mention two non-trivial results on  bounds for $\dim_k
H^1(H,M)$. First, the next theorem gives a general bound.

\begin{theorem}[Guralnick-Hoffman~{\cite[Thm. 1]{GH98}}]\label{GHbound} 
Let $H$ be  a finite group, $k$  a field, and $M$ be  a faithful irreducible
$kH$-module. Then
\[ \dim_k H^1(H,M) \leq \frac{1}{2} \dim_k M.\]
\end{theorem}

In special cases much better bounds are known, for example:

\begin{theorem}[Guralnick-Tiep~{\cite[Thm. 1.3]{GT11}}]\label{GTbound}
Let $\F_q$ be  a finite field with  $q$ elements and $H(q)$  a finite simple
Chevalley group of twisted rank $e$ over  $\F_q$ with Weyl group $W$ ($e$ is
at most the rank $l$ of the ambient  algebraic group). Let $k$ be a field of
characteristic different from the characteristic  of $\F_q$. Then the number
of irreducible $kH(q)$-modules $M$ with non-trivial $H^1(H(q),M)$ is bounded
in terms of $|W|$ and for all such $M$
\[ \dim_k H^1(H(q), M) \leq |W| + e.\]
\end{theorem}
Note that for fixed  $W$ and $e$ with growing $q$  the number of irreducible
$kH(q)$-modules and their  dimensions also grow, while the  given bounds are
independent of $q$.

For   a    long   time    all   explicitly    known   examples    of   first
cohomology   had   very   small    dimension.   Guralnick   formulated   the
following   conjecture,   see~\cite[Conj.~2]{G86}  and~\cite[Conj.~2]{GH98}.
\begin{conj}[Guralnick]\label{GConj} There exists a global constant $C$ such
that for  any finite  group $H$,  field $k$, and  $M$ a  finite dimensional,
absolutely irreducible, faithful $kH$-module we have
\[  \dim_k H^1(H,M) \leq C. \]
\end{conj}

It  was even  asked  if $C$  could  be  $2$. But  then  examples with  $H^1$
of  dimension  $3$ were  discovered  by  Scott~\cite{Scott03} and  Bray  and
Wilson~\cite{BW08}.  The  examples by  Scott  involved  the  computation  of
Kazhdan-Lusztig  polynomials  for  the  group $\SL_6(\bar\F_p)$  and  yields
examples for infinitely many finite  subgroups. Using the same technique for
$\SL_7(\bar\F_p)$ Scott and Sprowl even found examples of $H^1$ of dimension
$5$ (unpublished, in 2012).

\subsection{Connection of  {\boldmath $\dim H^1(H,M)$}  with Kazhdan-Lusztig
polynomials.}
The   mentioned   examples      by   Scott et~al.\   use   the   fact   that
certain  coefficients  of  the   Kazhdan-Lusztig  polynomials  described  in
Section~\ref{s:KLapp} have an interpretation  as $\dim H^1(H,M)$. To explain
this in more detail we use  the setup from Section~\ref{s:KLapp}. If $\Galg$
is a simply-connected reductive group  in characteristic $p$, defined over a
finite  field  $\Fq$,  we  consider  as finite  group  $H  =  \Galg(q)$  the
corresponding finite subgroup  of $\Fq$-rational points in  $\Galg$, and $M$
is  the restriction  of  a simple  module $L(\lambda)$  of  $\Galg$ to  $H$,
where  $\lambda$ is  a  dominant  $p$-restricted weight.  Note  that $M$  is
by~\cite[1.3]{St63} a simple module for the finite group $H$.

The  connection  between  Kazhdan-Lusztig   polynomials  and  dimensions  of
first  cohomology  groups is  established  by  combining the  following  two
theorems. The  first is  on algebraic  groups and we  use the  notation from
Section~\ref{s:KLapp}, see~\cite{And86} or~\cite[App C.2, C.10]{Jan03}.

\begin{theorem}[Andersen]\label{thm:KLextdim}
Let $\Galg$ be a connected simply-connected algebraic group over a field $k$
of characteristic $p$  where $p$ is a prime such  that the character formula
in Theorem~\ref{LuCF} holds. Let $W$ be  the Weyl group of $\Galg$ and $W_p$
the corresponding affine Weyl group. Let $x,  y \in W_p$ such that $x.0$ and
$y.0$ are dominant and $x.0$ is $p$-restricted. Then the coefficients of the
Kazhdan-Lusztig polynomial $m_{y,x}(v)$ can be interpreted as follows.

\[ m_{y,x}(v) = \sum_{i > 0} 
\dim {\Ext}^i_\Galg(V(y.0), L(x.0)) v^i.\]
\end{theorem}

In  particular, for  $y=1$,  that is  $y.0=0$, and $V(y.0)$  is the  trivial
module,  we   see  that  $\dim   H^1(\Galg,  L(x.0))$  is   the  coefficient
at  $v$  of  the   Kazhdan-Lusztig  polynomial  $m_{1,x}(v)$.  If  $\lambda$
is  a  $p$-restricted  weight  which  is  not  in  the  $W_p$-orbit  of  the
$0$-weight, then  $\dim H^1(\Galg,  L(x.0)) = 0$  by the  linkage principle,
see~\cite[II.6]{Jan03}.

Note that  in the  last theorem  the notions  of $H^1$  and ${\Ext}^i_\Galg$
are  with  respect to  homomorphisms  of algebraic  groups  (not  just  with
respect to  abstract groups as  in the definition above).  Nevertheless, the
${\Ext}$-spaces for  algebraic groups can  be related to  ${\Ext}$-spaces of
finite groups. For example, a  special case of~\cite[Thm. 7.5]{BNP01} yields
for large enough $p$ also dimensions of $H^1$ for finite groups.

\begin{theorem}[Bendel, Nakano, Pillen]
In the setup  of the previous theorem  let $p \geq 5(h-1)$ where  $h$ is the
Coxeter number  of $W$. We  assume that $\Galg$  is defined over  $\F_q$ and
write $H = \Galg(q)$. Then
\[ H^1(H,L(x.0)) \cong H^1(\Galg, L(x.0)). \]
\end{theorem}

(The statement  of~\cite[Thm. 7.5]{BNP01}  also addresses  higher cohomology
and other weights, the statement becomes more complicated for weights not in
the $W_p$-orbit of $0$.)

The examples of  $H^1$ of dimensions $3$ and $5$  by Scott et.al.\ mentioned
above came  from Kazhdan-Lusztig polynomials  for $\Galg$ of type  $A_5$ and
$A_6$, respectively.

In   view   of   the   Conjecture~\ref{GConj}   it   was   a   surprise   to
find   from   the   computed  Kazhdan-Lusztig   polynomials   described   in
Section~\ref{ss:Computations} examples of $H^1$ of much higher dimension. We
collect some interesting cases in the next theorem.

The coefficient of the Kazhdan-Lusztig polynomial in type $A_8$ mentioned in
the following theorem was independently computed by Scott and Sprowl
in~\cite{SS16} (more precisely, its congruence class modulo $2^{64}$).

\begin{theorem}\label{maxdimH}
Let $\Galg$  be defined  over $\F_q$  and the  characteristic $p$  of $\F_q$
be  at  least  $5(h-1)$,  where  $h$  is the  Coxeter  number  of  the  root
system  of $\Galg$,  and assume  that  in characteristic  $p$ the  character
formula~\ref{LuCF}  holds for  $\Galg$.  Let $\lambda$  be a  $p$-restricted
weight for $\Galg$ for which $\dim H^1(\Galg(q),L(\lambda))$ becomes maximal.

The following table shows $\dim H^1(\Galg(q),L(\lambda))$ for $\Galg$ simple
of small  rank (for smaller rank  simple groups all $H^1$  are of dimensions
$0$ or $1$):
\begin{center}
\begin{tabular}{l|ccccccc}
\hline
Type of $\Galg$ & $A_5$ & $A_6$ & $A_7$ & $A_8$ & $B_4$ & $B_5$ & $B_6$\\
$\dim H^1(\Galg(q),L(\lambda))$ & 3 & 16 & 469 & 36672 & 4 & 387 & 383868 \\
\hline
\end{tabular}
\end{center}

\begin{center}
\begin{tabular}{l|ccccccc}
\hline
Type of $\Galg$ & $C_5$ & $C_6$ & $D_4$ & $D_5$ & $D_6$ & $F_4$ & $E_6$\\
$\dim H^1(\Galg(q),L(\lambda))$ & 
                  130 & 259810 & 2 & 60 & 9504 & 882 & 3537142 \\
\hline
\end{tabular}
\end{center}
\end{theorem}

\noindent\textbf{Remarks.}\nopagebreak \begin{itemize}
\item[(a)] 
The  large coefficients  mentioned in  this table  occur in  Kazhdan-Lusztig
polynomials $m_{1,x}(v)$ where $x.0$ is close to the Steinberg weight. These
are among the computationally most  expensive polynomials for the respective
types. (So far, we could not finish  the computations for type $E_7$, but we
have already found dimensions of $H^1$  much larger than the largest in type
$E_6$ for weights which were still far from the Steinberg weight.)
\item[(b)]
The dimensions  mentioned for types $B_6$,  $C_6$ and $E_6$ are  much larger
than the order of  the Weyl group $W$ of $\Galg$, this  shows that the cross
characteristic assumption in Theorem~\ref{GTbound} is necessary.
\item[(c)]
The  dimensions  mentioned for  types  $A_7$,  $A_8$, $B_5$,  $B_6$,  $C_5$,
$C_6$,  $D_5$,  $D_6$,  $F_4$  and  $E_6$  are  larger  than  the  dimension
of  $\Galg$  as an algebraic variety. This  leads  to  the  construction  of
groups with  a huge  number of  maximal subgroups, as  will be  explained in
Section~\ref{s:WallConj}.
\item[(d)] 
Of course,  Conjecture~\ref{GConj} is  not disproved by  the finite  list of
examples given here. But the numbers  suggest to investigate the growth rate
of the maximal dimension of some $H^1$ in terms of the rank $l$ of the group
$\Galg$. From the numbers given in Theorem~\ref{maxdimH} even an exponential
growth in $l$ does not seem unlikely.

In~\cite{PSt14} it is shown that for $\Galg$ in any  characteristic and  any
dominant weight  $\lambda$ the  logarithm of this  maximal dimension  has an
upper bound of order $O(l^3\log(l))$.

Using   the  same   argument   as   in~\cite[5.1,5.2]{PSt14}  and   assuming
that    $\Galg$   and    $p$   are    as   in    Theorem~\ref{maxdimH}   our
estimate~\ref{KLcoeffbound}   yields   that  for   dominant   $p$-restricted
$\lambda$ we get the slightly stronger bound
\[ \dim H^1(\Galg, L(\lambda)) \leq 2^m, \]
where   $m$    is   the   length    of   the   longest   element    in   the
set   $B$   from~\ref{ss:Computations};   we    always   have   $m<3   l^3$,
see~\cite{Boe98},~\cite[3.3]{PSt14}.

For   the   groups   and   modules   considered   here,   the   bound   from
Theorem~\ref{GHbound}  could  be  improved  in~\cite[4.2.1]{BBDNPPW14},  the
factor  $\frac{1}{2}$ can  be  changed  to $\frac{1}{h}$  where  $h$ is  the
Coxeter number of the root system of $\Galg$, which grows linearly with $l$.
For small $p$ with respect to $l$ this may yield a better bound. The largest
dimension of any module $V(\lambda)$  with $\lambda$ a $p$-restricted weight
is $p^N$  where $N$  is the number  of positive roots  of $\Galg$,  this $N$
grows quadratically with $l$, we have for $l>8$:
\[\dim H^1(\Galg, L(\lambda)) \leq \frac{1}{l}p^N < \frac{1}{l}
p^{l^2}.\]

\item[(e)] 
We give  more details  for one example:  In case of  type $F_4$  the maximal
$\dim H^1(\Galg(q),L(\lambda)) = 882$  occurs for $\lambda = (p-2,p-2,p-2,9)
=  x.0$   for  an  $x$   of  length   $85$.  Using  the   Lusztig  character
formula~\ref{LuCF}, we can compute

{\footnotesize 
\[
\begin{split}
\hspace{1em} & {\dim} L(\lambda)  = 
\frac{677249}{19600}p^{21}-\frac{637702548041}{10478160000}p^{20}+\frac{4207651317557}{8382528000}p^{19}\\
& -\frac{64863221539889}{30177100800}p^{18} +\frac{99959647171}{22579200}p^{17}-\frac{72811406375072711}{1005903360000}p^{16}
\\ &
+\frac{5429608760885159}{33530112000}p^{15}
-\frac{1014183606287771}{5029516800}p^{14}+\frac{394068519721127}{1117670400}p^{13}
\\ &
-\frac{311337743406684289}{1508855040000}p^{12}-\frac{29048331062847}{13798400}p^{11}-\frac{19644400527431509}{10059033600}p^{10}
\\ &
+\frac{148882266983459651}{6706022400}p^{9}-\frac{20075834974540708571}{1005903360000}p^{8}+\frac{90508002252050021}{3048192000}p^{7}
\\ &
-\frac{503999924098905349}{3772137600}p^{6}+\frac{110387862657924361}{279417600}p^{5}-\frac{21129205276719213797}{20956320000}p^{4}
\\ &
+\frac{2067338407751272429}{698544000}p^{3}-\frac{80711097402731773}{11642400}p^{2}+\frac{137315811881887}{13860}p-5616745816.
\\ &
\end{split}
\]}
\end{itemize}

\section{Wall's conjecture}\label{s:WallConj}

For a finite group  $G$ we write $m(G)$ for the  number of maximal subgroups
of $G$.

The following conjecture was stated in 1962.

\begin{conj}[Wall~\cite{Wall62}] \label{WC}
For any finite group $G$ we have $m(G) < |G|$.
\end{conj}

It is clear that there cannot be a smaller bound since for $G$ an elementary
abelian $2$-group $m(G) =  |G|-1$. In this example there is  a big number of
maximal subgroups of small index.

Other examples of groups with many  maximal subgroups are dihedral groups of
order $2p$ for a  prime $p$. In that case the tiny  Sylow subgroups of order
$2$ are all maximal and there are $p$ of them, so that $m(G) > |G|/2$.

More  generally   Wall  showed   in~\cite{Wall62}  that  the   statement  in
Conjecture~\ref{WC}  is true  for $G$  a solvable  group. Some  better upper
bounds for solvable groups $G$ in  terms of the prime factorization of $|G|$
are also known, see for example~\cite{New11} and~\cite{HM95}.

The best known upper bounds for non-abelian simple groups and general finite
groups are given in the following theorem. Its proof uses the classification
of finite simple groups and many non-trivial results about maximal subgroups
of simple groups, see~\cite{LPS07}.

\begin{theorem}[Liebeck-Pyber-Shalev]\label{LPSWall}
There exists an absolute constant $C$ such that
\begin{itemize}
\item[(a)] for any almost simple group $G$  with socle a finite group of Lie
type  of rank  $l$ we  have $m(G)  < C  \cdot \left(\frac{1}{l}\right)^{2/3}
|G|$,
\item[(b)] and for any finite group $G$ we have $m(G) < 2 C |G|^{3/2}$.
\end{itemize}
\end{theorem}

\noindent\textbf{Remarks.} 
\begin{itemize}
\item[(a)]  Theorem~\ref{LPSWall}(a)  shows  that   a  weaker  form  of  the
statement in Wall's conjecture (with a  constant times $|G|$ as upper bound)
holds for  Lie type simple  groups and the  original statement holds  if the
rank is large enough.
\item[(b)] The exponent  $3/2 = 1 + \frac{1}{2}$  in ~\ref{LPSWall}(b) comes
from considering  configurations $G =  M \rtimes H$,  where $M$ is  a finite
irreducible module  for a  finite group  $H$, and  estimating the  number of
complements of  $M$ in  $G$ using  $H^1(H,M)$ as  in~\ref{ComplClasses}; the
$\frac{1}{2}$ term comes from the estimate in~\ref{GHbound}.
\end{itemize}

We now show that (also the weaker  form of) Wall's conjecture is not true by
constructing examples of type $G = M  \rtimes H$ where $H$ is a finite group
of Lie type and $M$ is a  finite module constructed from certain examples in
Theorem~\ref{maxdimH}. To explain how to  find appropriate finite modules we
need the following proposition.

\begin{proposition}\label{fqmod}
Let  $\Galg$  be  a  connected reductive  algebraic  group  over  $\bar\F_q$
as  in  Section~\ref{s:KLapp}, defined  over  $\F_q$,  and $H=\Galg(q)$  the
corresponding finite subgroup. For $p={\textrm char}(\F_q)$ let $\lambda$ be
a  $p$-restricted  weight for  $\Galg$  and  $L(\lambda)$ the  corresponding
simple module.
\begin{itemize}
\item[(a)] 
The restriction of $L(\lambda)$ to the finite group $H$ is simple.
\item[(b)] $L(\lambda)$  as $\bar\F_qH$-module can be  realized over $\F_q$,
that is  there is  an irreducible  $\F_qH$-module $M$  with $L(\lambda)  = M
\otimes_{\F_q} \bar\F_q$.
\item[(c)] If  $\lambda \neq 0$ then  the $\F_qH$-module $M$ in~(b)  is also
irreducible as $\F_pH$-module.
\end{itemize}
\end{proposition}
\begin{proof}
Parts~(a) and~(b) are shown in~\cite[1.3 and 7.5]{St63}.

For~(c) we  use~\cite[1.16 e)]{HuII} which  says that $M$ is  irreducible as
$\F_pH$-module  if  and only  if  the  character  field  of $M$  (the  field
generated by the traces of all $h \in  H$ on $M$) is $\F_q$. We have to show
that the character  field of $M$ is  $\F_q$. Let $q = p^f$  and let $\sigma$
be  the field  automorphism  of  $\bar\F_q$ that  raises  elements to  their
$p$-th power  and so  has $\F_p$  as fixed  field. Concatenating  the module
action of  $\Galg$ on $L(\lambda)$  with the field  automorphisms $\sigma^0,
\ldots, \sigma^{f-1}$ yields the Frobenius  twists of $L(\lambda)$ which are
$L(\lambda), L(p\lambda), \ldots, L(p^{f-1}\lambda)$.  Since we have assumed
that $\lambda$ is $p$-restricted, all of these twists are $q$-restricted and
since  $\lambda  \neq  0$  they  are  pairwise  different.  Therefore  their
restrictions to $H$ are all simple and pairwise non-isomorphic by~\cite[6.1,
7.4]{St63}.  This  shows  that  the  character field  of  $M$  cannot  be  a
proper subfield of $\F_q$.  (Non-isomorphic irreducible representations have
linearly independent characters, see~\cite[1.11(b)]{HuII}.)
\end{proof}

Now we can show the following theorem.

\begin{theorem}\label{cexWall}
There is a  real number $\delta >  0$ such that there  exist infinitely many
finite groups $G$ with
\[ m(G) > |G|^{1+\delta}. \]
\end{theorem}

\begin{proof}
Fix a prime $p>55$ such  that the Lusztig character formula~\ref{LuCF} holds
for the simple algebraic group $\Galg = F_4(\bar\F_p)$. Let $\lambda$ be the
$p$-restricted weight  of $\Galg$  as in  Theorem~\ref{maxdimH}, that  is we
have $\dim_{\bar\F_p} H^1(\Galg, L(\lambda)) = 882$.

For  each   power  $q$   of  $p$   let  $H_q   =  \Galg(q)$.   According  to
Proposition~\ref{fqmod}(b) there  is a  finite $\F_q H_q$-module  $M_q$ such
that $L(\lambda) = M_q \otimes_{\F_q} \bar\F_q$. Consider
\[G_q := M_q \rtimes H_q.\]

Then $H_q$ and any other complement of  $M_q$ in $G_q$ is a maximal subgroup
of  $G_q$. Indeed,  if  $x \in  G_q  \setminus H_q$,  then  $\langle H_q,  x
\rangle$ contains  a non-trivial  element in $M_q$  and, since  according to
Proposition~\ref{fqmod}(c)  $M_q$ is  an irreducible  $\F_pH_q$-module, also
contains an  $\F_p$-basis of  $M_q$, that  is a generating  set of  $M_q$ as
abelian group.

Using the definition of $H^1$ in the introduction of Section~\ref{s:GurConj}
we see that $H^1(H_q,  M_q)$ can be computed in terms of  a system of linear
equations with coefficients in $\F_q$. We conclude that
\[ \dim_{\F_q} H^1(H_q, M_q) = \dim_{\bar\F_q} H^1(H_q, L(\lambda)).\]
From  Theorem~\ref{maxdimH} we  know  that the  latter  dimension is  $882$.
Hence, by    Lemma~\ref{ComplClasses}, there are $q^{882}$ conjugacy classes
of complements of $M_q$ in $G_q$ and so we see that $G_q$ has at least
\[q^{882} |M_q|\]
maximal subgroups.

For the order of $G_q$  we have (see~\cite[2.9]{Ca}) 
\[|G_q| = |M_q| |H_q| = |M_q| q^{24}(q^2-1)(q^6-1)(q^8-1)(q^{12}-1)
< q^{52} |M_q|.\]

Now  let  $d  =  \dim_{\F_q}M_q =  \dim_{\bar\F_q}  L(\lambda)$.  Note  that
$d$  only  depends   on  our  fixed  $p$,  as  given   in  Remark~(e)  after
Theorem~\ref{maxdimH}, and not on $q$. Set $\delta = \frac{830}{d+52}$. Then
we have
\[ m(G_q) > q^{882+d} = (q^{52+d})^{1+\delta} > |G_q|^{1+\delta}\]
for all powers $q$ of $p$.  
\end{proof}

The  $\delta$  found  in  the  proof of~\ref{cexWall}  is  larger  when  the
characteristic is  small. If for  example the Lusztig character  formula was
valid  for type  $F_4$  and  $p=59$, we  would  get  $\delta \cong  1.6\cdot
10^{-36}$.  Fiebig's proven  upper bound  in~\cite{F12} for  the primes  for
which the Lusztig character formula may not be valid is not easy to evaluate
explicitly,  but it  is so  huge that  the resulting  $\delta$ is  extremely
small.  (The bound  involves the  maximal value  $r$ of  the Kazhdan-Lusztig
polynomials  computed   for  type  $F_4$   in  Section~\ref{ss:Computations}
evaluated at $1$ (which is $r=74628593$) and expressions containing a factor
$(r!)^r$.)

So, what is the minimal number $\varepsilon$ such that for all finite groups
$G$ we have  $m(G) < A \cdot |G|^{1+\varepsilon}$ with  some global constant
$A$? From Theorem~\ref{cexWall} and Theorem~\ref{LPSWall} we know
\[ 0 < \varepsilon < \frac{1}{2}.\]
With the method we used to  prove Theorem~\ref{cexWall} applied to groups of
rank  $l$ we  can  never  show that  $\varepsilon  >  \frac{1}{l}$, even  if
the  upper bounds  for dimensions  of  $H^1$ mentioned  in Remark~(d)  after
Theorem~\ref{maxdimH} are achieved.
\bigskip

\bibliographystyle{plain}
\bibliography{KLPolyGuralnickWallConj}

\end{document}